\documentclass[12pt]{amsart}
\usepackage[top=30truemm,bottom=30truemm,left=25truemm,right=25truemm]{geometry}
\usepackage{mathrsfs}
\usepackage{amsmath, amsthm, amssymb}
\usepackage{mathtools}
\usepackage{color}
\usepackage{bm}
\usepackage{amsfonts}
\usepackage{dsfont}
\usepackage{amscd}
\usepackage{extarrows}
\usepackage{enumerate} 
\usepackage[all]{xy}
\usepackage[colorlinks]{hyperref}
\usepackage[nameinlink, capitalize, noabbrev]{cleveref}
\usepackage[colorlinks]{hyperref}

\definecolor{darkblue}{RGB}{0, 102, 204}   

\hypersetup{
    pdfencoding=auto,
    colorlinks=true,    
    linkcolor=darkblue, 
    citecolor=darkblue, 
    urlcolor=darkblue,  
    filecolor=darkblue  
}
\newtheorem{theorem}{Theorem}[section]
\newtheorem{definition}{Definition}[section]
\newtheorem{lemma}{Lemma}[section]
\newtheorem{corollary}{Corollary}[section]
\newtheorem{proposition}{Proposition}[section]
\newtheorem{remark}{Remark}[section]

\newtheorem{conjecture}{Conjecture}[section]

\newcommand{\be}{\begin{equation}}
	\newcommand{\ee}{\end{equation}}
\newcommand{\bea}{\begin{eqnarray}}
	\newcommand{\eea}{\end{eqnarray}}
\newcommand{\ben}{\begin{eqnarray*}}
	\newcommand{\een}{\end{eqnarray*}}
\newcommand{\bt}{\begin{split}}
	\newcommand{\et}{\end{split}}
\newcommand{\bet}{\begin{equation}}

	\newcommand{\ddbar}{\partial \bar{\partial}}

	\numberwithin{equation}{section}
\begin{document}

\title[The Chern-Ricci flow on minimal surfaces of general type]
{Convergence of the Chern-Ricci flow on complex minimal surfaces of general type}

\author[H. Sun]{Haoyuan Sun}
\address{Haoyuan Sun: School of Mathematical Sciences\\ Beijing Normal University\\ Beijing 100875\\ P. R. China}
\email{202531130037@mail.bnu.edu.cn}

\begin{abstract}
We prove uniform diameter estimates, volume non-collapsing estimates and Gromov-Hausdorff convergence for the
normalized Chern-Ricci flow on smooth complex minimal surfaces of general type,
starting from an arbitrary Hermitian metric. This removes the local K\"ahler assumption near the null locus used in our previous work and confirms the Tosatti-Weinkove conjecture in complex dimension two.
The main analytic
ingredients are a surface torsion estimate, a uniform total variation bound for
$\Delta |G|$, a Green-weighted $L^2$ estimate for the torsion, and a linear iteration
of real Poisson equations, which together give the required Green function estimates.
\end{abstract}



\maketitle
\section{Introduction}

The long-time behavior of the K\"ahler-Ricci flow is a central topic in complex
geometry and in the analytic Minimal Model Program of Song and Tian
\cite{ST07,ST12,ST17}. One of the most difficult analytic problems in this direction is
to establish uniform geometric estimates when the flow degenerates. In particular,
uniform diameter estimates play a crucial role in understanding the metric limit of
canonical metrics and of the K\"ahler-Ricci flow. In the K\"ahler setting, this problem
has been studied extensively; see for example
\cite{SW13,Song14,GSW16,TZ16,CW17,Guo17,Wang18,CW19,STZ19,CW20,GPSS23,LT23,GPSS24a,GPSS24b,GTô25,GGZ25,Vu26,Sun26}
and the references therein.

Recently, Guo, Phong, Song, and Sturm developed a powerful Green function method
for proving diameter bounds for families of K\"ahler metrics with entropy control
\cite{GPSS23,GPSS24a,GPSS24b}. Their work builds on new estimates for Green functions and
complex Monge-Amp\`ere equations in \cite{GPS24,GPT23,GPTW24}, and gives a robust
framework which avoids relying on pointwise Ricci lower bounds. This framework is
one of the main inspirations for the present paper.

The Hermitian analogue of the K\"ahler-Ricci flow is the Chern-Ricci flow introduced
by Tosatti and Weinkove \cite{TW13,TW15}. Given a compact complex manifold $X$ and an
initial Hermitian metric $\omega_0$, the flow is
\begin{equation}\label{main CRF}
    \frac{\partial}{\partial s}\widetilde\omega(s)=-Ric^C(\widetilde\omega(s)),
\qquad
\widetilde\omega(0)=\omega_0,
\end{equation}
where $Ric^C$ is the Chern-Ricci form. If $K_X$ is nef, the flow exists for all time.
After the standard change of variables $t=\log(1+s)$ and
$\omega_t=(1+s)^{-1}\widetilde\omega(s)$, the normalized Chern-Ricci flow becomes
\begin{equation}\label{eq:intro-normalized-crf}
\frac{\partial}{\partial t}\omega_t=-Ric^C(\omega_t)-\omega_t,
\qquad
\omega_t|_{t=0}=\omega_0.
\end{equation}
We recall the following conjecture of Tosatti and Weinkove.

\begin{conjecture}{\cite[Conjecture 4.1]{TW22}} \label{conjecture}
Let $X^n$ be a compact complex manifold with $K_X$ nef and big. Let
$\widetilde\omega(s)$ be the solution of the Chern-Ricci flow \eqref{main CRF} starting at an arbitrary
Hermitian metric $\omega_0$. Then
\[
\operatorname{diam}\left(X,\frac{\widetilde\omega(s)}{s+1}\right)\le C
\]
for all $s$ sufficiently large, and
\[
\left(X,\frac{\widetilde\omega(s)}{s+1}\right)\longrightarrow (Z,d)
\]
in the Gromov-Hausdorff topology for some compact metric space $(Z,d)$. Moreover,
if $E=\text{Null}(K_X)$, then $(Z,d)$ should be identified with the metric completion of
$(X\setminus E,\omega_{KE})$, where $\omega_{KE}$ is the singular K\"ahler-Einstein
current obtained as the limit of the normalized Chern-Ricci flow.
\end{conjecture}

When the initial metric is K\"ahler, this conjecture is known in several cases. On
surfaces of general type, the conjecture was proved independently by
Guo--Song--Weinkove \cite{GSW16} and Tian--Zhang \cite{TZ16}; the latter also
treated the K\"ahler case in complex dimension three. Under the additional assumption
of a uniform Ricci lower bound, Guo proved the conjecture for arbitrary projective manifolds of
general type \cite{Guo17}. The K\"ahler case in arbitrary dimension was completely
settled by Wang \cite{Wang18}. Recently, Jian--Song \cite{JS26} gives an alternative proof of the Gromov--Hausdorff convergence, and Lee--Tosatti-Zhang \cite{LTZ26} further establish the convergence even in the collapsing case. In the Hermitian setting, the problem is substantially
harder because the evolving metrics are not closed and the torsion of the Chern
connection enters the Green formula through first-order drift terms.

In our previous work \cite{Sun26}, the conjecture was proved under the additional
assumption that the initial metric is K\"ahler in a neighborhood of the null locus
$E=\text{Null}(K_X)$ (but in arbitrary dimensions). This hypothesis eliminates the torsion terms in the degenerating
region and allows one to adapt the Green function estimates from the K\"ahler setting.
The purpose of the present paper is to remove this local K\"ahler assumption in
complex dimension two.

Our main result is as follows.

\begin{theorem}\label{thm:main}
Let $X$ be a smooth compact complex minimal surface of general type and let
$\omega_0$ be an arbitrary Hermitian metric on $X$. Let $\omega_t$ be the normalized
Chern-Ricci flow \eqref{eq:intro-normalized-crf}. 
Then there exist uniform constants $C>0$, $c>0$, $\alpha>0$ and $r_0>0$ such that
\[
\operatorname{diam}(X,\omega_t)\le C
\]
for all $t\ge0$, and
\[
\operatorname{Vol}_{\omega_t}\big(B_{\omega_t}(x,r)\big)\ge cr^\alpha
\]
for every $x\in X$, every $0<r<r_0$, and every $t\ge0$. Moreover, $(X,\omega_t)$ converges to $\overline{(X\setminus E,\omega_{KE})}$ in the Gromov-Hausdorff topology as $t\to+\infty$. In particular, \cref{conjecture} holds true in dimension $2$.
\end{theorem}

Let us now explain the main analytic difficulties and the new ideas in the proof.
Let $\Delta_{g_t}$ be the real Laplacian associated to $\omega_t$, and let
$\Delta_{\omega_t}$ be the Chern Laplacian. If $\theta_t$ denotes the real torsion
one-form of $\omega_t$, then
\[
\Delta_{g_t}f
=
2\Delta_{\omega_t}f
+
2\langle\nabla f,\theta_t^\#\rangle_{\omega_t}.
\]
Thus, when one applies the real Green formula to equations involving the Chern
Laplacian, one has to control drift terms of the form
\[
\int_X G_t(x,y)\langle\nabla u,\theta_t^\#\rangle_{\omega_t}\,\omega_t^2.
\]
These terms are absent in the K\"ahler setting and are the main obstruction to applying
the method of \cite{GPSS24a,GPSS24b} directly.

The first new observation is a surface torsion estimate (see \cref{lem:surface-torsion-estimates} below). Since
$\omega_t=\hat\omega_t+\sqrt{-1}\partial\bar\partial\varphi_t$ and
$\hat\omega_t=\chi+e^{-t}(\omega_0-\chi)$ with $\chi$ closed, we have
$\partial\omega_t=e^{-t}\partial\omega_0$. In complex dimension two, a pointwise
eigenvalue computation gives
\[
|\theta_t|_{\omega_t}^2\le Ce^{-2t}\operatorname{tr}_{\omega_t}\omega_0.
\]
Consequently,
\[
\int_X|\theta_t|_{\omega_t}^2\,\omega_t^2\le Ce^{-2t}.
\]
Moreover, choosing $\chi\ge0$, we obtain the stronger pointwise estimate
$|\theta_t|_{\omega_t}^2\le Ce^{-t}\operatorname{tr}_{\omega_t}\hat\omega_t$.
This is the only place where the surface assumption is used in an essential way.

The second ingredient is a Green-weighted torsion estimate (\cref{rmk:kato-torsion-input-surface} below). We prove first that the
distributional Laplacian of $|G_t(x,\cdot)|$ has uniformly bounded total variation.
This is a Kato-type argument for the positive and negative parts of the Green function,
in the spirit of Brezis--Ponce \cite{BP04}. Combining this with the surface torsion
estimate and a weighted energy identity, we obtain
\[
\sup_{z\in X}\int_X |G_t(z,y)|\,|\theta_t(y)|_{\omega_t}^2\,\omega_t^2(y)
\le Ce^{-t}.
\]
This estimate can be seen as the analytic replacement for the local K\"ahler assumption in
\cite{Sun26}. It allows us to absorb all first-order torsion terms appearing in the
Green function estimates.

The third new ingredient is a linear iteration method based
on real Poisson equations. Starting from a smooth nonnegative function $F$, we solve
successively
\[
\Delta_{g_t}q_i=2(h_i-\bar h_i),
\qquad
h_{i+1}=\langle\nabla q_i,\theta_t^\#\rangle_{\omega_t}.
\]
The Green-weighted torsion estimate gives
\[
D_t(h_{i+1})\le Ce^{-t/2}D_t(h_i),
\]
so the iteration converges in the Green norm. Summing the series produces a smooth
function $Q_F$ satisfying
\[
\Delta_{\omega_t}Q_F=F-c_F,
\qquad
\operatorname{osc}_XQ_F+|c_F|\le CD_t(F).
\]
This allows us to prove the Laplacian estimate needed for the Green lower bound in
the presence of the first-order torsion term.

With these tools in hand, we establish the three Green function estimates required by
the method of Guo--Phong--Song--Sturm. Namely, if $G_t$ is the real Green function normalized by
$\int_XG_t(x,y)\omega_t^2(y)=0$, and if
$\mathcal G_t(x,y)=G_t(x,y)-\inf_zG_t(x,z)+V_t^{-1}$, then
\[
\sup_x\int_X|G_t(x,y)|\,\omega_t^2(y)\le C,
\qquad
\inf_{x,y}G_t(x,y)\ge -C,
\]
and for some uniform $\varepsilon>0$,
\[
\sup_x\int_X\mathcal G_t(x,y)^{1+\varepsilon}\,\omega_t^2(y)\le C.
\]
The last estimate is obtained by combining the drifted Laplacian estimate with
auxiliary complex Monge-Amp\`ere equations, following the philosophy of
\cite{GPS24,GPSS24a,Sun26}.

After the Green function estimates are established, the rest of the proof is purely
Riemannian. The $L^{1+\varepsilon}$ bound for $\mathcal G_t$ gives an
$L^s$ bound for $|\nabla G_t|$, for some $s>1$. Applying the Green representation
formula to distance functions then gives the uniform diameter estimate. A standard
cutoff argument with distance functions gives the volume non-collapsing estimate.

It remains to identify the Gromov-Hausdorff limit. It was already shown by Song \cite{Song14} that $\overline{(X\setminus E,\omega_{KE})}$ is homeomorphic to the canonical model $(X_{can},d_{can})$ of $X$, since in dimension $2$ we have that $X$ is projective.
By Gromov's precompactness theorem, the diameter and volume estimates give sequential Gromov-Hausdorff convergence. Moreover,
the volume non-collapsing estimate, together with the uniform upper bound
$\omega_t^2\le C\Omega$, rules out the formation of extra limit points over arbitrarily
small neighborhoods of $E$. Finally, since $f(E)$ is a finite set, one can compare the
distance $d_t$ with the quotient distance obtained from $d_{can}$ by collapsing small
neighborhoods of the orbifold points, where we have again used the dimension $2$ assumption essentially. Letting the size of these neighborhoods tend to
zero gives the distance comparison needed to identify every subsequential limit with
$(Y,d_{can})$. This proves the full Gromov-Hausdorff convergence.

\begin{remark}\label{rmk:higher-dimensional-outlook}
Let us also mention a possible higher dimensional replacement for the local
K\"ahler assumption in \cite{Sun26}. Let $\psi\le0$ be a quasi-plurisubharmonic
function with analytic singularities along $\text{Null}(K_X)$ such that
\[
\chi+\sqrt{-1}\partial\bar\partial\psi\ge c\omega_0,\quad c>0,
\]
and suppose that the standard singular lower bound along the normalized
Chern-Ricci flow in \cite[Lemma 3.5]{Gill13} takes the form
\[
\omega_t\ge C^{-1}e^{B\psi}\omega_0.
\]
In a local $\omega_0$-unitary frame diagonalizing $\omega_t$, the torsion term is
controlled by
\[
|\theta_t|_{\omega_t}^2
\le
Ce^{-2t}\sum_{\ell\ne j}
\lambda_\ell^{-1}\lambda_j^{-2}|A^0_{\ell j\bar j}|^2,
\]
where $A^0$ denotes the tensor determined by $\partial\omega_0$. Hence, if we assume that the
initial torsion vanishes along the null locus to the order
\[
|\partial\omega_0|_{\omega_0}^2\le C e^{(3B+\delta)\psi}
\]
for some $\delta>0$, then the above singular lower bound gives
\[
|\theta_t|_{\omega_t}^2\le Ce^{-2t}e^{\delta\psi}\le Ce^{-2t}.
\]
In particular,
\[
\int_X|\theta_t|_{\omega_t}^2\,\omega_t^n\le Ce^{-2t},
\qquad
\int_X|G_z|\,|\theta_t|_{\omega_t}^2\,\omega_t^n
\le
Ce^{-2t}\int_X|G_z|\,\omega_t^n .
\]
Thus, under this finite order vanishing assumption, the (Green weighted) $L^2$ torsion terms are exponentially small (repeat the proof of \cref{prop:surface-L1-green}) and the arguments in this paper could be carried over in higher dimensions. This
condition is much weaker than assuming that $\omega_0$ is K\"ahler in a
neighborhood of the null locus $E$, which corresponds to the vanishing of
$\partial\omega_0$ there. 
\end{remark}

The paper is organized as follows. In \cref{sec:preliminaries}, we recall the
basic estimates for the Chern-Ricci flow and the relation between the real and complex
Laplacians. In \cref{section 3}, we prove the surface torsion estimates and the Laplacian
estimates needed later. In \cref{section 4}, we establish the $L^1$ bound, the lower bound,
and the $L^{1+\varepsilon}$ estimate for the real Green function. In \cref{section 5}, we prove the diameter and volume non-collapsing estimates.
Finally, in \cref{section 6}, we identify the Gromov-Hausdorff limit with
the canonical model endowed with its orbifold K\"ahler-Einstein distance.

\section{Preliminaries}\label{sec:preliminaries}

In this section we collect some standard facts and preliminary estimates for the
Chern-Ricci flow which will be used throughout the paper. We mostly follow the
notation and conventions in \cite{Gill13,TW15,Sun26}.

Let $(X,\omega_X)$ be a compact Hermitian manifold with $K_X$ nef and big. We first recall the definition of null locus:
\begin{definition}\label{def: Null locus}
    The null locus $\text{Null}(K_X)$ of the big line bundle $K_X$ is defined to be the union of all positive-dimensional irreducible analytic subvarieties $V\subset X$ such that if $\operatorname{dim}_{\mathbb{C}}V=k$, then
    $$
\int_V(c_1^{BC}(K_X))^k=0.
    $$
When the underlying manifold $X$ is K\"ahler (which is true in our case $n=2$ since minimal surfaces of general type are projective) or lies in the Fujiki class $\mathcal{C}$, it was shown by Collins-Tosatti \cite[Theorem 1.1]{CT15} that $\text{Null}(K_X)$ coincides with the non-K\"ahler locus of $K_X$ and hence is itself a proper analytic subvariety of $X$. When $X$ is a general Hermitian manifold, some progress towards this issue was made by Dang \cite{Dang24}.   
\end{definition}

Let $\omega_0$ be an arbitrary Hermitian metric on $X$. We consider the Chern-Ricci
flow
\[
\frac{\partial}{\partial s}\widetilde\omega(s)
=
-Ric^C(\widetilde\omega(s)),
\qquad
\widetilde\omega(0)=\omega_0.
\]
Equivalently, after the change of variables $t=\log(1+s)$ and
$\omega_t=(1+s)^{-1}\widetilde\omega(s)$, we obtain the normalized Chern-Ricci flow
\begin{equation}\label{eq:normalized-CRF}
\frac{\partial}{\partial t}\omega_t
=
-Ric^C(\omega_t)-\omega_t,
\qquad
\omega_{t}|_{t=0}=\omega_0.
\end{equation}
Here
\[
Ric^C(\omega_t)
=
-\sqrt{-1}\partial\bar\partial\log\omega_t^n
\]
is the Chern-Ricci form.

Let $\chi$ be a smooth representative of $-c_1^{BC}(X)$ and choose a smooth volume form
$\Omega$ such that $\chi=\sqrt{-1}\partial\bar\partial\log\Omega$. As in \cite{Gill13, Sun26}, $\chi$ can be chosen to be semi-positive and such that $\chi<\omega_0$ after a rescaling of $\omega_0$. We write $\hat\omega_t:=\chi+e^{-t}(\omega_0-\chi)$. The normalized flow can be written as $\omega_t=\hat{\omega}_t+\sqrt{-1}\ddbar\varphi_t$,
where $\varphi_t$ solves
\begin{equation}\label{eq:potential-equation}
\frac{\partial\varphi_t}{\partial t}
=
\log\frac{(\hat\omega_t+\sqrt{-1}\partial\bar\partial\varphi_t)^n}{\Omega}
-\varphi_t,
\qquad
\varphi_0=0.
\end{equation}
We shall use the notation
\[
\dot\varphi_t:=\frac{\partial\varphi_t}{\partial t},
\qquad
\mu_t:=\omega_t^n,
\qquad
V_t:=\int_X\mu_t.
\]
We next recall the volume notions for nef Bott-Chern classes introduced in \cite{GL22, BGL25}. If $\beta$ is a nef real $(1,1)$-form, its lower volume is defined by
\[
\underline{\text{Vol}}(\{\beta\})
:=
\lim_{\varepsilon\to0^+}
\inf_{u\in C^\infty(X),\,\beta+\varepsilon\omega_X+\sqrt{-1}\partial\bar\partial u>0}
\int_X(\beta+\varepsilon\omega_X+\sqrt{-1}\partial\bar\partial u)^n.
\]
The positive volume property will be used through the following consequence, which was recently established by Pang--Sun--Wang--Zhou \cite{PSWZ25}, generalizing the corresponding results of Guo--Phong--Tong--Wang \cite{GPT23, GPTW24} on K\"ahler manifolds.

\begin{theorem}\cite[Theorem 15.4]{PSWZ25} \label{a priori in nef class for MA}
    Let $\{\beta\}\in BC^{1,1}(X)$ be a nef Bott-Chern class satisfying $\underline{\operatorname{Vol}}(\{\beta\})>0$. Assume also $\varphi_t\in \operatorname{PSH}(X,\chi+t\omega_X)\cap C^\infty(X)$ satisfying
$$
(\beta+t\omega_X+dd^c\varphi_t)^n= c_te^{F_t}\omega_X^n,\quad\sup_X\varphi_t=0.
$$
Here $F_t\in C^\infty(X)$. We also fix a constant $p>n$. Then there exists a uniform constant $C$ depending on $\chi,\omega_X,n,p$, the upper bound of $\int_{X}e^{F_t(z)}[\log(1+e^{F_t(z)})]^p\omega_X^n$, the lower bound of $\int_Xe^{\frac{F_t}{n}}\omega_X^n$ and the lower bound of  $\underline{\operatorname{Vol}}(\{\beta\})$ such that
$$
0\leq-\varphi_t+\mathcal{V}_t\leq C.
$$
Here $\mathcal{V}_t:=\sup\{u\;|\;u\in \operatorname{PSH}(X,\beta+t\omega_X),u\leq0\}$ is the largest non-positive $(\beta+t\omega_X)$-plurisubharmonic function.
\end{theorem}

The following estimates are established by Gill \cite{Gill13} and then further generalized by Sun \cite{Sun26} using \cref{a priori in nef class for MA}. 

\begin{lemma}\cite{Gill13, Sun26} \label{lem:basic-flow-estimates}
Along the normalized Chern-Ricci flow, there exists a uniform constant $C>0$ such that
\[
|\varphi_t|\le C,\qquad |\dot\varphi_t|\le C,
\]
and hence
\[
C^{-1}\omega_X^n\le \omega_t^n\le C\omega_X^n.
\]
Moreover, $\omega_t$ converges as currents to a closed positive current $\omega_{KE}$,
which is smooth on $X\setminus E$, where $E=\text{Null}(K_X)$, and
\[
\omega_t\longrightarrow \omega_{KE}
\quad\text{in }C^\infty_{\mathrm{loc}}(X\setminus E).
\]
\end{lemma}

We will repeatedly use the relation between the real and complex Laplacians on a Hermitian
manifold. Let $\omega=\sqrt{-1}g_{i\bar j}dz^i\wedge d\bar z^j$ be a Hermitian metric.
The complex Laplacian is
\[
\Delta_\omega f=g^{i\bar j}\partial_i\partial_{\bar j}f,
\]
while the real Laplacian of the associated Riemannian metric $g$ is denoted by $\Delta_g$.

Let $\tau=\Lambda_\omega(\partial\omega)$ be the trace torsion $(1,0)$-form and set $\theta:=\tau+\bar\tau$.
In local coordinates, we have
\[
\tau=
\left(
g^{j\bar k}\partial_\ell g_{j\bar k}
-
g^{j\bar k}\partial_jg_{\ell\bar k}
\right)dz^\ell.
\]
Equivalently, $\tau$ is characterized by
\[
\tau\wedge\omega^{n-1}=\partial\omega^{n-1}.
\]
We refer the reader to \cite[Lemma 3.1]{Sun26} for more details. The following lemma will be used frequently:
\begin{lemma}\label{lem:real-complex-laplacian}
For every $f\in C^2(X)$,
\[
\Delta_g f
=
2\Delta_\omega f
+
2\langle df,\theta\rangle_\omega
=
2\Delta_\omega f
+
2\langle\nabla f,\theta^\#\rangle_\omega.
\]
\end{lemma}

\section{Laplacian estimates}\label{section 3}

We first recall the definition and normalization of the Green function. For the real
Laplacian $\Delta_{g_t}$, the Green function $G_t(x,y)$ is normalized by
\begin{equation}\label{eq:green-normalization}
\Delta_{g_t,y}G_t(x,y)\,\mu_t(y)
=
-\delta_x+\frac1{V_t}\mu_t(y),
\qquad
\int_XG_t(x,y)\,\mu_t(y)=0.
\end{equation}
Thus, for every $u\in C^2(X)$,
\begin{equation}\label{eq:real-green-formula}
u(x)
=
\frac1{V_t}\int_Xu\,\mu_t
-
\int_XG_t(x,y)\Delta_{g_t}u(y)\,\mu_t(y).
\end{equation}
We shall often write $G_z(y):=G_t(z,y)$.

Recall that we can choose a semi-positive representative $\chi$ of $K_X$ and set $\hat{\omega}_t:=\chi+e^{-t}(\omega_0-\chi)$, where $\omega_0$ is the initial metric for the normalized Chern-Ricci flow such that $\omega_0>\chi$. As in \cite{Sun26}, we may use \cite[Theorem 4.6]{GL22} to obtain that $\underline{\operatorname{Vol}}(\chi)>0$. Consequently, we easily have that the volumes $V_{\omega_t}:=\int_X\omega_t^n$ along the flow is uniformly bounded away from zero.

For this reason, we will define the $p$-Nash Entropy of $\omega_t$ by
$$
\operatorname{Ent}_p(\omega_t):=\int_X\left|\log\left(\frac{\omega_t^n}{\omega_X^n}\right)\right|^p\omega_t^n=\int_Xe^{F_t}|F_t|^p\omega_X^n,
$$
where $e^{F_t}:=\frac{\omega_t^n}{\omega_X^n}$. It then follows easily from \cref{lem:basic-flow-estimates} and the integrability of quasi-plurisubharmonic functions that $\operatorname{Ent}_p(\omega_t)$ remains uniformly bounded along the flow.

The following Laplacian estimate is taken from \cite[Lemma 4.1]{Sun26}. 

\begin{lemma}\label{lem:laplacian estimate 1}
Let $v\in L^1(X,\mu_t)$ and assume that
\[
v\in C^2(\overline{\Omega_0}),
\qquad
\Delta_{\omega_t}v\ge -a
\quad\text{on }\Omega_0,
\]
where $\Omega_s:=\{v>s\}$ for each $s\geq0$ and $a>0$ is a given constant. Then, there is a constant $C>0$ depending on $n,p,\chi,\underline{\operatorname{Vol}}(\chi),\omega_X,\operatorname{Ent}_p(\omega_t)$ such that
\[
\sup_Xv
\le
C\left(a+\int_Xv_+\,\mu_t\right).
\]
\end{lemma}

\begin{remark}
In the proof of \cite[Lemma 4.1]{Sun26}, the normalization $\|v\|_{L^1}\le1$ is used only to control the integral
\[
A_s=\int_{\Omega_s}(v-s)\,\mu_t.
\]
The same argument works verbatim if this is replaced by
$\int_Xv_+\,\mu_t\le1$ and the condition $\int_Xv\mu_t=0$ is not needed.
\end{remark}

In dimension $2$, we have the following surface torsion estimates:
\begin{lemma}\label{lem:surface-torsion-estimates}
Along the normalized Chern-Ricci flow, we have $|\theta_t|_{\omega_t}^2\le Ce^{-2t}\operatorname{tr}_{\omega_t}\omega_0$. Since $\chi$ is also semi-positive, we have $|\theta_t|_{\omega_t}^2\le Ce^{-t}\operatorname{tr}_{\omega_t}\hat\omega_t$. Consequently,
\[
\int_X|\theta_t|_{\omega_t}^2\,\mu_t\le Ce^{-2t}.
\]
\end{lemma}

\begin{proof}
Since $\partial\chi=0$ and $\partial(\sqrt{-1}\partial\bar\partial\varphi_t)=0$, we have
\[
\partial\omega_t=e^{-t}\partial\omega_0,\qquad
\tau_t=e^{-t}\Lambda_{\omega_t}(\partial\omega_0).
\]
At a fixed point, choose an $\omega_0$-unitary frame such that
\[
\omega_t=\sqrt{-1}\lambda_1dz^1\wedge d\bar z^1
+\sqrt{-1}\lambda_2dz^2\wedge d\bar z^2.
\]
Writing $A^0_{ij\bar k}=\partial_i(g_0)_{j\bar k}-\partial_j(g_0)_{i\bar k}$, we get
\[
(\tau_t)_1=e^{-t}\lambda_2^{-1}A^0_{12\bar2},\qquad
(\tau_t)_2=-e^{-t}\lambda_1^{-1}A^0_{12\bar1}.
\]
Thus
\[
|\tau_t|_{\omega_t}^2
\le
Ce^{-2t}\left(\lambda_1^{-1}\lambda_2^{-2}
+\lambda_2^{-1}\lambda_1^{-2}\right).
\]
By the uniform volume equivalence \cref{lem:basic-flow-estimates}, we have
$C^{-1}\le\lambda_1\lambda_2\le C$. Hence
\[
\lambda_1^{-1}\lambda_2^{-2}\le C\lambda_2^{-1},\qquad
\lambda_2^{-1}\lambda_1^{-2}\le C\lambda_1^{-1},
\]
and therefore
\[
|\tau_t|_{\omega_t}^2\le Ce^{-2t}(\lambda_1^{-1}+\lambda_2^{-1})
=Ce^{-2t}\operatorname{tr}_{\omega_t}\omega_0.
\]
Since $|\theta_t|_{\omega_t}^2\le C|\tau_t|_{\omega_t}^2$, the pointwise estimate follows.

Integrating it gives
\[
\int_X|\theta_t|_{\omega_t}^2\,\mu_t
\le
Ce^{-2t}\int_X\operatorname{tr}_{\omega_t}\omega_0\,\omega_t^2
=
Ce^{-2t}\int_X\omega_0\wedge\omega_t.
\]
The last integral is uniformly bounded, because
\[
\int_X\omega_0\wedge\omega_t
=
\int_X\omega_0\wedge\hat\omega_t
+
\int_X\varphi_t\,\sqrt{-1}\partial\bar\partial\omega_0
\le C,
\]
using the uniform boundedness of $\varphi_t$.

Finally, if $\chi\ge0$, then $\hat\omega_t=e^{-t}\omega_0+(1-e^{-t})\chi\ge e^{-t}\omega_0$, hence
\[
\operatorname{tr}_{\omega_t}\omega_0\le e^t\operatorname{tr}_{\omega_t}\hat\omega_t.
\]
The final estimate follows.
\end{proof}

Using \cref{lem:surface-torsion-estimates}, the proof of \cite[Lemma 4.2]{Sun26} carries over easily to give the following $L^1$-estimate of Laplacian solutions. We record the proof here for the reader's convenience.
\begin{lemma}\label{Laplacian estimate 2}
    Suppose that $v\in C^2(X)$ satisfies 
    $$
|\Delta_{\omega_t}v|\leq 1,\quad \int_Xv\omega_t^n=0.
    $$
    Then, there exists a uniform constant $C=C(n,p,\chi,\underline{\operatorname{Vol}}(\chi),\omega_X,\operatorname{Ent}_p(\omega_t))$ such that
$$
\int_X|v|\omega_t^n\leq C.
$$
\end{lemma}
\begin{proof}
       We will mainly follow the blow-up arguments in \cite[Lemma 5.3]{GPSS24a}. Suppose that there is a sequence of metrics $\omega_{t_j}$ along the CRF (with $t_j\to+\infty$) and $v_j\in C^2(X)$ such that 
    $$
\Delta_{\omega_{t_j}}v_j=h_j,\quad\int_Xv_j\omega_{t_j}^2=0,
    $$
    where $h_j$ are continuous functions satisfying $|h_j|\leq1$, and we have
    $$
\int_X|v_j|\omega_{t_j}^2:=N_j\to+\infty,
    $$
    as $j\to+\infty$. Set $\hat{v}_j:=\frac{v_j}{N_j}$, then we have
    \[
|\Delta_{\omega_{t_j}}\hat{v}_j|=\frac{|h_j|}{N_j}\to0,\quad\int_X|\hat{v}_j|\omega_{t_j}^2=1.
    \]
\cref{lem:laplacian estimate 1} then yields a uniform constant $C$ such that
\begin{equation}\label{eq:sup vj}
\sup_X|\hat{v}_j|\leq C.
\end{equation}
Let $\tau_j:=\Lambda_{\omega_{t_j}}(\partial{\omega_{t_j}})=e^{-{t_j}}\Lambda_{\omega_{t_j}}(\partial{\omega_{0}})$ be the torsion $(1,0)$-form of $\omega_{t_j}$. Using \cref{lem:surface-torsion-estimates}, we can write 
\begin{align*}
    \int_X|\nabla\hat{v}_j|^2_{\omega_{t_j}}\omega_{t_j}^2&\leq C\int_X\sqrt{-1}\partial\hat{v}_j\wedge\bar{\partial}\hat{v}_j\wedge\omega_{t_j}\\
    &=\frac{C}{n}\int_X(-\hat{v}_j)(\Delta_{\omega_{t_j}}\hat{v}_j)\omega_{t_j}^2+C\int_X(-\hat{v}_j)\bar{\partial}\hat{v}_j\wedge\partial\omega_{t_j}\\
    &\leq C\int_X|\hat{v}_j|\left|\frac{h_j}{N_j}\right|\omega_{t_j}^2+C\int_X(-\hat{v}_j)\bar{\partial}\hat{v}_j\wedge\tau_j\wedge\omega_{t_j}\\
    &\leq C\int_X|\hat{v}_j|\left|\frac{h_j}{N_j}\right|\omega_{t_j}^2+C\int_X|\nabla\hat{v}_j|_{\omega_{t_j}}|\tau_j|_{\omega_{t_j}}\omega_{t_j}^2\\
    &\leq C\int_X|\hat{v}_j|\left|\frac{h_j}{N_j}\right|\omega_{t_j}^2+C\left(\int_X|\nabla\hat{v}_j|^2_{\omega_{t_j}}\omega_{t_j}^2\right)^{\frac{1}{2}}\left(\int_X|\tau_j|_{\omega_{t_j}}^2\omega_{t_j}^2\right)^\frac{1}{2},
\end{align*}
where in the third inequality we have used \eqref{eq:sup vj}. Since $\left(\int_X|\tau_j|_{\omega_{t_j}}^2\omega_{t_j}^2\right)^\frac{1}{2}=O(e^{-t_j})$ and $\left|\frac{h_j}{N_j}\right|\to0$ uniformly, we conclude that 
\begin{equation}\label{eq:L2 of nabla v_j}
\underset{j\to\infty}{\lim}  \int_X|\nabla\hat{v}_j|^2_{\omega_{t_j}}\omega_{t_j}^2\to0.
\end{equation}
Since $X$ is a connected complex manifold and $E=\text{Null}(K_X)$ is a proper analytic subvariety, it is well-known that $E$ is automatically thin in $X$ and hence $X\setminus E$ is connected. For each $\varepsilon>0$, we can thus find small connected neighborhoods $U_\varepsilon\Subset U_{2\varepsilon}$ with smooth boundaries such that $X\setminus U_\varepsilon,X\setminus U_{2\varepsilon}$ are connected and that $|U_\varepsilon|<\varepsilon,|U_{2\varepsilon}|<2\varepsilon$ (here we use $|\cdot|$ to denote the volume of a set with respect to $\omega_X^2$).

Now Since $\omega_{t_j}$ are uniformly equivalent outside $U_\varepsilon$, we conclude from \eqref{eq:L2 of nabla v_j} that
\begin{equation}\label{eq:L2 limit of vj 2}
    \underset{j\to\infty}{\lim}\int_{X\setminus U_{\varepsilon}}|\nabla\hat{v}_j|^2_{\omega_X}\omega_X^2=0.
\end{equation}
It follows that the sequence $\hat{v}_j$ is uniformly bounded in the sobolev space $W^{1,1}(X\setminus U_{\varepsilon},\omega_X^2)$ and hence we can use sobolev's embedding theorem to extract a converging subsequence (still denoted by $\hat{v}_j$) such that $\hat{v}_j\to v_\infty$ in $L^q(X\setminus U_{2\varepsilon},\omega_X^2)$ for some $q>1$. By \eqref{eq:sup vj} we have that $\hat{v}_j$ is uniformly bounded, thus so does $v_\infty$. Recall that we have smooth convergence $\omega_{t_j}\to\omega_{KE}$ on $X\setminus U_{2\varepsilon}$, this clearly yields uniform convergence of the coefficients of the formal adjoint operators $\Delta_{\omega_{t_j}}^*\to\Delta_{\omega_{KE}}^*$ on $X\setminus U_{2\varepsilon}$. Consequently, we easily have the convergence 
$$\frac{h_j}{N_j}=\Delta_{\omega_{t_j}}\hat{v}_j\to\Delta_{\omega_{KE}}v_\infty$$
on $X\setminus U_{2\varepsilon}$ in the distributional sense. Therefore, we have $\Delta_{\omega_{KE}}v_\infty=0$ on $X\setminus U_{2\varepsilon}$ and hence $v_\infty$ is smooth on $X\setminus U_{2\varepsilon}$ by standard elliptic regularity theory.

Now we claim that $v_\infty$ is constant in $X\setminus U_{2\varepsilon}$. For each smooth vector field $Y\in C_c^{\infty}(X\setminus U_{2\varepsilon})$, we use the divergence theorem for the underlying Riemannian manifold to write
\begin{align*}
    \int_X\langle\nabla v_\infty,Y\rangle_{\omega_X}\omega_X^2&=-\int_Xv_\infty\cdot\operatorname{div}_{\omega_X}Y\omega_X^2=-\underset{j\to\infty}{\lim}\int_X\hat{v}_j\operatorname{div}_{\omega_X}Y\omega_X^2\\
  &  =\underset{j\to\infty}{\lim} \int_X\langle\nabla \hat{v}_j,Y\rangle_{\omega_X}\omega_X^2\leq C\underset{j\to\infty}{\lim} \int_{X\setminus U_{\varepsilon}}|\nabla\hat{v}_j|_{\omega_X}\omega_X^2=0.
\end{align*}
Here, we have used \eqref{eq:L2 limit of vj 2} in the last equality. This immediately yields that $\nabla v_\infty\equiv0$ on $X\setminus U_{2\varepsilon}$, whence the claim follows.

It remains to establish a contradiction. Write $\hat{v}_j=\hat{v}_j^+-\hat{v}_j^-$ and set $c=v_\infty$ on $X\setminus U_{2\varepsilon}$.

\textbf{Case 1.} Suppose first $c\geq0$. Since $\int_X\hat{v}_j\omega_{t_j}^2=0$, we have $\int_X\hat{v}_j^+\omega_{t_j}^2=\int_X\hat{v}_j^-\omega_{t_j}^2$. By \eqref{eq:sup vj} and \cref{lem:basic-flow-estimates} we have
$$
\int_{U_{2\varepsilon}}\hat{v}_j^{-}\omega_{t_j}^2\leq C|U_{2\varepsilon}|\leq 2C\varepsilon.
$$
Since $c\geq0$, $\hat{v}_j^{-}:=-\min(\hat{v}_j,0)\to0$ on $X\setminus U_{2\varepsilon}$. The dominated convergence theorem then yields that
$$
\underset{j\to\infty}{\lim}\int_{X\setminus U_{2\varepsilon}}\hat{v}_j^{-}\omega_{t_j}^2=0.
$$
We now choose $\varepsilon$ so small such that $2C\varepsilon<\frac{1}{4}$, then for $j$ large
$$
1=\int_X|\hat{v}_j|\omega_{t_j}^2=\int_X(\hat{v}_j^-+\hat{v}_j^+)\omega_{t_j}^2=2\int_X\hat{v}_j^-\omega_{t_j}^2=2\int_{U_{2\varepsilon}}\hat{v}_j^-\omega_{t_j}^2+2\int_{{X\setminus U_{2\varepsilon}}}\hat{v}_j^-\omega_{t_j}^2<\frac{1}{2},
$$
this gives a contradiction.

\textbf{Case 2.} For the case $c<0$, the argument is exactly the same as in \textbf{Case 1} by considering $\hat{v}_j^+$, the proof is therefore concluded.
\end{proof}

Combining this with \cite[Corollary 4.1]{Sun26} (the local K\"ahler assumption there is superfluous), we get
\begin{corollary}\label{cor:Lap estimate final}
    Suppose that $v\in C^2(X)$ satisfies 
    $$
|\Delta_{\omega_t}v|\leq 1,\quad \int_Xv\omega_t^n=0.
    $$
    Then, there exists a uniform constant $C=C(n,p,\chi,\underline{\operatorname{Vol}}(\chi),\omega_X,\operatorname{Ent}_p(\omega_t))$ such that
    $$
\|v\|_{L^\infty}\le C.
    $$
\end{corollary}

\section{Green function estimates on Hermitian surfaces}\label{section 4}

\subsection{$L^1$ estimates of the Green function}

We begin with a distributional estimate that the total variation norm (denote by $\|\cdot\|_{\mathcal{M}}$) of the real Laplacian of Green functions is uniformly bounded with respect to $t$, closely related to the Kato-type inequalities of Brezis--Ponce \cite{BP04}. 

\begin{lemma}\label{lem:laplacian-abs-green}
Let $G=G_t(x,\cdot)$ be the real Green function and set $G_+=\max\{G,0\}$, $G_-=\max\{-G,0\}$.
Then, in the sense of distributions,
\[
\Delta_{g_t}G_+\,\mu_t
=
-\delta_x+\frac1{V_t}\mathbf 1_{\{G>0\}}\mu_t+\nu_+,
\]
and
\[
\Delta_{g_t}G_-\,\mu_t
=
-\frac1{V_t}\mathbf 1_{\{G<0\}}\mu_t+\nu_-,
\]
where $\nu_+,\nu_-$ are nonnegative Radon measures satisfying $\nu_+(X)\le1$ and $\nu_-(X)\le1$.
Consequently, we have the Laplacian bounds of $|G|$ under the Total variation norm:
\[
\|\Delta_{g_t}|G|\mu_t\|_{TV}\le C.
\]
\end{lemma}

\begin{proof}
Choose smooth convex functions $\beta_j:\mathbb R\to\mathbb R_{\ge0}$ such that
\[
\beta_j(s)=0\quad\text{for }s\le0,\qquad
\beta_j(s)=s-a_j\quad\text{for }s\ge 2j^{-1},
\]
where $0\le a_j\le j^{-1}$, and
\[
0\le\beta_j'\le1,\qquad \beta_j''\ge0,\qquad
0\le s_+-\beta_j(s)\le j^{-1}.
\]
In particular, $\beta_j(G)\to G_+$ in $L^1(X,\mu_t)$ and
$\beta_j'(G)\to\mathbf 1_{\{G>0\}}$ almost everywhere and in $L^1(X,\mu_t)$.

We first treat $G_+$. Since $G(y)\to+\infty$ as $y\to x$, for each fixed $j$
there exists a neighborhood $U_j$ of $x$ such that $G\ge2j^{-1}$ on
$U_j\setminus\{x\}$. Thus
\[
\beta_j(G)=G-a_j,\qquad \beta_j'(G)=1,\qquad \beta_j''(G)=0
\]
on $U_j\setminus\{x\}$. Hence $\beta_j(G)$ has exactly the same Dirac singularity
as $G$ at $x$.

On $X\setminus\{x\}$, the usual chain rule gives
\[
\Delta\beta_j(G)
=
\beta_j'(G)\Delta G+\beta_j''(G)|\nabla G|_{g_t}^2
=
\frac1V\beta_j'(G)
+
\beta_j''(G)|\nabla G|_{g_t}^2.
\]
Because $\beta_j(G)=G-a_j$ near the pole, this identity extends across $x$ as
the distributional identity
\[
\Delta\beta_j(G)\,\mu
=
-\delta_x+\frac1V\beta_j'(G)\mu+\nu_{+,j},
\]
where
\[
\nu_{+,j}:=\beta_j''(G)|\nabla G|_{g_t}^2\mu
\]
is a nonnegative Radon measure.

Testing against $1$ gives
\[
0=-1+\frac1{V_t}\int_X\beta_j'(G)\,\mu_t+\nu_{+,j}(X),
\]
so $\nu_{+,j}(X)\le1$. Passing to a subsequence, $\nu_{+,j}$ converges weakly
to a nonnegative Radon measure $\nu_+$ with $\nu_+(X)\le1$. Since
$\beta_j'(G)\to\mathbf 1_{\{G>0\}}$ in $L^1(X,\mu_t)$ and
$\beta_j(G)\to G_+$ in distributions, we obtain
\[
\Delta_{g_t}G_+\,\mu_t
=
-\delta_x+\frac1{V_t}\mathbf 1_{\{G>0\}}\mu_t+\nu_+.
\]
The proof for $G_-$ is the same, applied to $\beta_j(-G)$. Since $G$ is positive
near the pole, $\beta_j(-G)$ vanishes in a neighborhood of $x$, and therefore no
Dirac mass occurs. On $X\setminus\{x\}$,
\[
\Delta_{g_t}\beta_j(-G)
=
-\beta_j'(-G)\Delta_{g_t}G+\beta_j''(-G)|\nabla G|_{g_t}^2,
\]
and hence distributionally
\[
\Delta_{g_t}\beta_j(-G)\,\mu_t
=
-\frac1{V_t}\beta_j'(-G)\mu_t+\nu_{-,j},
\qquad
\nu_{-,j}:=\beta_j''(-G)|\nabla G|_{g_t}^2\mu_t\ge0.
\]
Testing against $1$ yields
\[
\nu_{-,j}(X)=\frac1{V_t}\int_X\beta_j'(-G)\,\mu_t\le1.
\]
After passing to a subsequence,
$\nu_{-,j}\rightharpoonup\nu_-$ with $\nu_-(X)\le1$, and
$\beta_j'(-G)\to\mathbf 1_{\{G<0\}}$ in $L^1(X,\mu_t)$. Thus
\[
\Delta_{g_t}G_-\,\mu_t
=
-\frac1{V_t}\mathbf 1_{\{G<0\}}\mu_t+\nu_-.
\]

Finally, since $|G|=G_++G_-$,
\[
\Delta_{g_t}|G|\,\mu_t
=
-\delta_x
+
\frac1{V_t}\left(\mathbf 1_{\{G>0\}}-\mathbf 1_{\{G<0\}}\right)\mu_t
+\nu_++\nu_-.
\]
The total variation of the right-hand side is uniformly bounded, which proves
the last assertion.
\end{proof}

The following weighted energy identity will be used frequently in the sequel:
\begin{lemma}\label{lem:weighted-energy-identity}
Let $(M,g)$ be a closed Riemannian manifold and let $\Delta=\Delta_g$ be the real Laplacian operator.
If $w\in C^\infty(M)$ and $f\in C^\infty(M)$, then
\[
\int_M w|\nabla f|^2\,d\mu
=
-\int_M wf\Delta f\,d\mu
+
\frac12\int_M f^2\Delta w\,d\mu.
\]
In particular, if $w=|G_t(x,\cdot)|$, the same identity holds in the distributional
sense:
\[
\int_X |G_t||\nabla f|_{g_t}^2\,\mu_t
=
-\int_X |G_t|f\Delta_{g_t} f\,\mu_t
+
\frac12\int_X f^2\Delta_{g_t}|G_t|\,\mu_t.
\]
\end{lemma}

\begin{proof}
For smooth $w$, integration by parts gives
\[
-\int_M wf\Delta f\,d\mu
=
\int_M\langle\nabla(wf),\nabla f\rangle\,d\mu
=
\int_Mw|\nabla f|^2\,d\mu
+
\int_Mf\langle\nabla w,\nabla f\rangle\,d\mu.
\]
On the other hand,
\[
\int_M f^2\Delta w\,d\mu
=
-\int_M\langle\nabla(f^2),\nabla w\rangle\,d\mu
=
-2\int_M f\langle\nabla f,\nabla w\rangle\,d\mu.
\]
Therefore
\[
\int_Mf\langle\nabla w,\nabla f\rangle\,d\mu
=
-\frac12\int_Mf^2\Delta w\,d\mu.
\]
Substituting this into the previous identity gives
\[
\int_M w|\nabla f|^2\,d\mu
=
-\int_M wf\Delta f\,d\mu
+
\frac12\int_M f^2\Delta w\,d\mu.
\]

For $w=|G_t(x,\cdot)|$, we argue by cutting off the pole first. Fix
$r>0$ and choose $\chi_r\in C^\infty(X)$ such that
\[
0\le \chi_r\le1,\qquad
\chi_r\equiv0\ \text{on }B_{g_t}(x,r),\qquad
\chi_r\equiv1\ \text{on }X\setminus B_{g_t}(x,2r).
\]
For $\varepsilon>0$, set
\[
w_{\varepsilon,r}:=\chi_r\left(G_t(x,\cdot)^2+\varepsilon^2\right)^{1/2}.
\]
Since $\chi_r$ vanishes near the pole, $w_{\varepsilon,r}$ is smooth on $X$.
Applying the smooth identity to $w_{\varepsilon,r}$ gives
\[
\int_X w_{\varepsilon,r}|\nabla f|_{g_t}^2\,\mu_t
=
-\int_X w_{\varepsilon,r}f\Delta_{g_t}f\,\mu_t
+
\frac12\int_X f^2\Delta_{g_t}w_{\varepsilon,r}\,\mu_t.
\]
Letting $\varepsilon\to0$, we have
\[
w_{\varepsilon,r}\to \chi_r|G_t(x,\cdot)|
\quad\text{in }L^1(X,\mu_t),
\]
and also in the sense of distributions. Hence
\[
\Delta_{g_t}w_{\varepsilon,r}
\to
\Delta_{g_t}\big(\chi_r|G_t(x,\cdot)|\big)
\]
distributionally. Therefore
\[
\int_X \chi_r|G_t|\,|\nabla f|_{g_t}^2\,\mu_t
=
-\int_X \chi_r|G_t|\,f\Delta_{g_t}f\,\mu_t
+
\frac12
\left\langle
\Delta_{g_t}\big(\chi_r|G_t|\big),f^2
\right\rangle.
\]
Finally let $r\to0$. Since $G_t(x,\cdot)$ is locally integrable near its pole,
$\chi_r|G_t|\to |G_t|$ in $L^1(X,\mu_t)$. Hence the first two terms converge by
dominated convergence. For the last term, the convergence is distributional:
for every smooth test function $\phi$,
\[
\left\langle
\Delta_{g_t}\big(\chi_r|G_t|\big),\phi
\right\rangle
=
\int_X\chi_r|G_t|\Delta_{g_t}\phi\,\mu_t
\longrightarrow
\int_X|G_t|\Delta_{g_t}\phi\,\mu_t
=
\left\langle\Delta_{g_t}|G_t|,\phi\right\rangle.
\]
Taking $\phi=f^2$ gives
\[
\int_X |G_t|\,|\nabla f|_{g_t}^2\,\mu_t
=
-\int_X |G_t|f\Delta_{g_t}f\,\mu_t
+
\frac12
\left\langle
\Delta_{g_t}|G_t|,f^2
\right\rangle.
\]
By \cref{lem:laplacian-abs-green}, $\Delta_{g_t}|G_t|$ is a Radon measure with
uniformly bounded total variation, so the last pairing is equivalently written as
\[
\frac12\int_X f^2\Delta_{g_t}|G_t|\,\mu_t.
\]
\end{proof}

\begin{proposition}\label{prop:surface-L1-green}
There is a uniform constants $C$ such that the real Green function $G_t$ satisfies
\[
\sup_{x\in X}\int_X|G_t(x,y)|\,\omega_t^2(y)\le C.
\]
\end{proposition}
\begin{proof}
It is enough to prove the estimate for all sufficiently large $t$. Fix $x\in X$, as usual we set
\[
G:=G_t(x,\cdot),\qquad \mu_t:=\omega_t^2,\qquad A:=\int_X|G|\,\mu_t.
\]
The goal is then to obtain a uniform bound of $A$ from above. Let
\[
K:=\int_X|G|\,|\theta_t|_{\omega_t}^2\,\mu_t,\qquad
M:=\int_X|G|\operatorname{tr}_{\omega_t}\hat\omega_t\,\mu_t.
\]
By \cref{lem:surface-torsion-estimates}, we have
\begin{equation}\label{eq:K and M}
K\le Ce^{-t}M.
\end{equation}
Since $\omega_t=\hat\omega_t+\sqrt{-1}\partial\bar\partial\varphi_t$, in complex
dimension two, we have $\Delta_{\omega_t}\varphi_t=2-\operatorname{tr}_{\omega_t}\hat\omega_t$. Hence 
\[
\operatorname{tr}_{\omega_t}\hat\omega_t
=
2-\frac12\Delta_{g_t}\varphi_t
+
\langle\nabla\varphi_t,\theta_t^\#\rangle_{\omega_t}.
\]
Thus
\begin{equation}\label{eq:M and P}
M\le C(A+1)+P,\qquad
P:=\int_X|G|\,|\nabla\varphi_t|_{\omega_t}|\theta_t|_{\omega_t}\,\mu_t.
\end{equation}
Indeed, by the distributional bound for $\Delta_{g_t}|G|$ and the uniform
boundedness of $\varphi_t$,
\[
\left|\int_X|G|\Delta_{g_t}\varphi_t\,\mu_t\right|
=
\left|\left\langle\Delta_{g_t}|G|,\varphi_t\right\rangle\right|
\le C.
\]
Set
\[
E_\varphi:=\int_X|G|\,|\nabla\varphi_t|_{\omega_t}^2\,\mu_t.
\]
By the Cauchy-Schwarz inequality, $P^2\le KE_\varphi$. The weighted energy identity applied with $f=\varphi_t$ \cref{lem:weighted-energy-identity} gives
\[
E_\varphi
\le
C\int_X|G|\,|\Delta_{g_t}\varphi_t|\,\mu_t+C.
\]
Since
\[
\Delta_{g_t}\varphi_t
=
4-2\operatorname{tr}_{\omega_t}\hat\omega_t
+
2\langle\nabla\varphi_t,\theta_t^\#\rangle_{\omega_t},
\]
we obtain
\[
E_\varphi\le C(A+M+P+1).
\]
Consequently,
\[
P^2\le KE_\varphi
\le Ce^{-t}M(A+M+P+1).
\]
Combining this with \eqref{eq:M and P} yields
\[
P^2\le Ce^{-t}(A+1+P)^2.
\]
This implies that for $t$ sufficiently large,
\[
P\le Ce^{-t/2}(A+1),
\]
which combined with \eqref{eq:M and P} and \eqref{eq:K and M} shows that
\begin{equation}\label{eq:M,K,A}
    M\le C(A+1),\qquad K\le Ce^{-t}(A+1).
\end{equation}
Choose smooth functions $h_k$ with $|h_k|\le2$ and
$h_k\to-\mathbf 1_{\{G\ge0\}}$ almost everywhere. Let $v_k$ solve
\[
\Delta_{\omega_t}v_k=h_k+c_k,\qquad \int_Xv_k\,\mu_t=0.
\]
The constant satisfies $|c_k|\le2$ (cf. \cite[Lemma 4.3, Lemma 4.4]{Sun26}). By \cref{cor:Lap estimate final}, we get
\[
\|v_k\|_{L^\infty}\le C.
\]
Define
\[
T_k:=\int_X|G|\,|\nabla v_k|_{\omega_t}|\theta_t|_{\omega_t}\,\mu_t,\qquad
E_k:=\int_X|G|\,|\nabla v_k|_{\omega_t}^2\,\mu_t.
\]
By the weighted energy identity \cref{lem:weighted-energy-identity} and the total variation bound for
$\Delta_{g_t}|G|$ \cref{lem:laplacian-abs-green},
\begin{equation}\label{eq:E_k and T_k}
    \begin{aligned}
    E_k&\le C\int_X|G|\,|\Delta_{g_t}v_k|\,\mu_t+C\\
    &=C\int_X|G|\,|2(h_k+c_k)+2\langle\nabla v_k,\theta_t^\#\rangle_{\omega_t}|\,\mu_t+C\\
    &\le C(A+1)+CT_k.
\end{aligned}
\end{equation}
On the other hand, $T_k^2\le KE_k$, hence \eqref{eq:E_k and T_k} and \eqref{eq:M,K,A} gives
\[
T_k^2\le Ce^{-t}(A+1)(A+1+T_k).
\]
It follows that, for all sufficiently large $t$,
\begin{equation}\label{eq:estimate of T_k}
    T_k\le Ce^{-t/2}(A+1).
\end{equation}
Since $\int_Xv_k\,\mu_t=0$, the real Green formula gives
\[
v_k(x)
=
-\int_XG\,\Delta_{g_t}v_k\,\mu_t
=
\int_XG(-2h_k)\,\mu_t
-
2\int_XG\langle\nabla v_k,\theta_t^\#\rangle_{\omega_t}\,\mu_t.
\]
The last term is bounded in absolute value by $2T_k$. Therefore
\begin{equation}\label{eq:hk and Tk}
    \int_XG(-2h_k)\,\mu_t
\le C+2T_k.
\end{equation}
Letting $k\to\infty$, using $G\in L^1(X,\mu_t)$ for fixed $t,x$, we obtain
\[
\int_XG(-2h_k)\,\mu_t
\longrightarrow
2\int_{\{G\ge0\}}G\,\mu_t.
\]
Since $\int_XG\,\mu_t=0$, we have $2\int_{\{G\ge0\}}G\,\mu_t=\int_X|G|\,\mu_t=A$. Thus \eqref{eq:estimate of T_k} and \eqref{eq:hk and Tk} yields
\[
A\le C+Ce^{-t/2}(A+1).
\]
For $t$ sufficiently large, the last term is absorbed into the left-hand side, so $A\le C$. This proves the desired uniform $L^1$ bound.
\end{proof}
\begin{remark}\label{rmk:kato-torsion-input-surface}
    From the proof of \cref{prop:surface-L1-green} (see \eqref{eq:K and M} and \eqref{eq:M,K,A}), we can also extract a very useful estimate: there is a uniform constant $C$ such that
    \begin{equation}
\kappa_t:=
\sup_{z\in X}
\int_X |G_z(y)|\,|\theta_t(y)|_{\omega_t}^2\,\mu_t(y)
\le Ce^{-t}.
\end{equation}
\end{remark}

\subsection{Lower bound for the Green function on Hermitian surfaces}

To avoid any regularity issue caused by the norm $|\theta_t|_{\omega_t}$ at its zero set, we introduce the smooth function
\[
\Theta_t:=\left(|\theta_t|_{\omega_t}^2+e^{-4t}\right)^{1/2}.
\]
Then it follows from \cref{lem:surface-torsion-estimates} and \cref{rmk:kato-torsion-input-surface} that
\begin{equation}\label{eq:Theta-properties}
\Theta_t\ge |\theta_t|_{\omega_t},
\qquad
\int_X\Theta_t^2\,\mu_t\le Ce^{-2t},
\qquad
\sup_{z\in X}\int_X|G_z|\,\Theta_t^2\,\mu_t\le Ce^{-t}.
\end{equation}
For any measurable function $f$, define
\[
m_t(f):=\int_X |f|\,\mu_t,\qquad
P_t(f):=\sup_{z\in X}\int_X |G_z(y)|\,|f(y)|\,\mu_t(y),
\qquad
D_t(f):=m_t(f)+P_t(f).
\]

\begin{lemma}\label{lem:real-poisson-green-norm}
Let $h\in C^\infty(X)$ and let $q_h$ solve
\begin{equation}\label{eq:real-poisson-qh}
\Delta_{g_t}q_h=2(h-\bar h),\qquad
\int_Xq_h\,\mu_t=0,
\end{equation}
where $\bar h=V_t^{-1}\int_Xh\,\mu_t$. Then there is a uniform constant $C$, independent of $t$, such that
\begin{equation}\label{eq:qh-osc-Dh}
\operatorname{osc}_Xq_h\le CD_t(h),
\end{equation}
and
\begin{equation}\label{eq:qh-weighted-energy}
\sup_{z\in X}\int_X|G_z|\,|\nabla q_h|_{\omega_t}^2\,\mu_t
+
\int_X|\nabla q_h|_{\omega_t}^2\,\mu_t
\le CD_t(h)^2.
\end{equation}
\end{lemma}

\begin{proof}
By Green's formula and \eqref{eq:real-poisson-qh},
\[
q_h(z)=-2\int_XG_z(y)(h(y)-\bar h)\,\mu_t(y).
\]
Using \cref{prop:surface-L1-green} and $|\bar h|\le CV_t^{-1}m_t(h)$, we get
\[
|q_h(z)|
\le
2\int_X|G_z||h|\,\mu_t
+
2|\bar h|\int_X|G_z|\,\mu_t
\le CD_t(h),
\]
which proves \eqref{eq:qh-osc-Dh}. Since $\int_Xq_h\,\mu_t=0$, this also gives
$\|q_h\|_{L^\infty}\le CD_t(h)$.

For the weighted energy estimate, \cref{lem:weighted-energy-identity} gives
\[
\int_X|G_z|\,|\nabla q_h|_{\omega_t}^2\,\mu_t
=
-\int_X|G_z|q_h\Delta_{g_t}q_h\,\mu_t
+
\frac12\int_Xq_h^2\Delta_{g_t}|G_z|\,\mu_t.
\]
The first term is bounded by
\[
C\|q_h\|_{L^\infty}
\left(
\int_X|G_z||h|\,\mu_t
+
|\bar h|\int_X|G_z|\,\mu_t
\right)
\le CD_t(h)^2,
\]
and the second by $C\|q_h\|_{L^\infty}^2$ using
\(\|\Delta_{g_t}|G_z|\|_{\mathcal M}\le C\). Taking the supremum in $z$ gives
the weighted part of \eqref{eq:qh-weighted-energy}. The unweighted estimate follows from
\[
\int_X|\nabla q_h|_{\omega_t}^2\,\mu_t
=
-\int_Xq_h\Delta_{g_t}q_h\,\mu_t
\le C\|q_h\|_{L^\infty}m_t(h)
\le CD_t(h)^2.
\]
\end{proof}
Next we construct a barrier function to eliminate $F$ in the linear Laplacian estimates containing drift terms.
\begin{lemma}\label{lem:linear-drift-iteration}
Let $F\ge0$ be a smooth function on $X$. For all sufficiently large $t$, there exist
$Q_F\in C^\infty(X)$ and $c_F\in\mathbb R$ such that
\begin{equation}\label{eq:linear-iteration-equation}
\Delta_{\omega_t}Q_F=F-c_F,
\qquad
\int_XQ_F\,\mu_t=0.
\end{equation}
Moreover,
\begin{equation}\label{eq:linear-iteration-bounds}
\operatorname{osc}_XQ_F+|c_F|\le CD_t(F).
\end{equation}
\end{lemma}

\begin{proof}
Set $h_0:=F$. Once $h_i$ is defined, let $q_i$ solve
\[
\Delta_{g_t}q_i=2(h_i-\bar h_i),\qquad
\int_Xq_i\,\mu_t=0,\qquad
\bar h_i:=V_t^{-1}\int_Xh_i\,\mu_t,
\]
and put
\[
h_{i+1}:=\langle\nabla q_i,\theta_t^\#\rangle_{\omega_t}.
\]
By \cref{lem:real-poisson-green-norm},
\begin{equation}\label{eq:osc qi}
    \operatorname{osc}_Xq_i\le CD_t(h_i),
\end{equation}
and
\[
\sup_z\int_X|G_z|\,|\nabla q_i|_{\omega_t}^2\,\mu_t
+
\int_X|\nabla q_i|_{\omega_t}^2\,\mu_t
\le CD_t(h_i)^2.
\]
Using \eqref{eq:Theta-properties} and the Cauchy-Schwarz inequality,
\[
P_t(h_{i+1})
\le
\left(\sup_z\int_X|G_z|\Theta_t^2\,\mu_t\right)^{1/2}
\left(\sup_z\int_X|G_z|\,|\nabla q_i|_{\omega_t}^2\,\mu_t\right)^{1/2}
\le Ce^{-t/2}D_t(h_i),
\]
and similarly
\[
m_t(h_{i+1})
\le
\left(\int_X\Theta_t^2\,\mu_t\right)^{1/2}
\left(\int_X|\nabla q_i|_{\omega_t}^2\,\mu_t\right)^{1/2}
\le Ce^{-t}D_t(h_i).
\]
Thus
\[
D_t(h_{i+1})\le Ce^{-t/2}D_t(h_i).
\]
For all sufficiently large $t$, the factor on the right is at most $1/2$; hence
\[
D_t(h_i)\le 2^{-i}D_t(F).
\]
It thus follows from \eqref{eq:osc qi} that the series $\sum_i q_i$ converges uniformly, since each $q_i$ has zero
$\mu_t$-average and hence $\|q_i\|_{L^\infty}\le \operatorname{osc}_Xq_i$.
Define
\[
Q_F:=\sum_{i=0}^\infty q_i,\qquad
c_F:=\sum_{i=0}^\infty \bar h_i.
\]
Then
\[
\operatorname{osc}_XQ_F\le C\sum_iD_t(h_i)\le CD_t(F),
\]
and
\[
|c_F|\le C\sum_i m_t(h_i)\le CD_t(F).
\]
It remains to check the equation. By definition we have
\[
\Delta_{\omega_t}q_i
=
h_i-\bar h_i-h_{i+1}.
\]
Therefore, for $N\ge0$,
\[
\Delta_{\omega_t}\sum_{i=0}^Nq_i
=
h_0-\sum_{i=0}^N\bar h_i-h_{N+1}.
\]
Since $D_t(h_{N+1})\to0$, we have $h_{N+1}\to0$ in $L^1(X,\mu_t)$. Passing to
the limit in the sense of distributions gives
\[
\Delta_{\omega_t}Q_F=F-c_F.
\]
The right-hand side is smooth, so standard elliptic regularity for the fixed
smooth operator $\Delta_{\omega_t}$ gives $Q_F\in C^\infty(X)$. The normalization
$\int_XQ_F\,\mu_t=0$ follows from the normalization of each $q_i$.
\end{proof}

We can now give the proof of the estimate of Poisson equations containing drift terms:
\begin{lemma}\label{lem:drifted-laplacian-estimate}
Let $u\in C^\infty(X)$ and $F\ge0$ be smooth. If 
\[
\Delta_{\omega_t}u\ge -a-F\quad\text{on }\{u>0\},
\]
then, for all sufficiently large $t$, there is a uniform constant $C$ such that
\[
\sup_Xu\le C\left(a+\int_Xu_+\,\mu_t+D_t(F)\right).
\]
\end{lemma}

\begin{proof}
Let $Q_F,c_F$ be given by \cref{lem:linear-drift-iteration}. Set
\[
\widetilde Q_F:=Q_F-\sup_XQ_F\le0,\qquad U:=u+\widetilde Q_F.
\]
Then $U\le u$, so $\{U>0\}\subset\{u>0\}$. On this set,
\[
\Delta_{\omega_t}U
\ge
-a-F+F-c_F
\ge
-a-|c_F|.
\]
By \cref{lem:laplacian estimate 1}, we deduce that
\[
\sup_XU\le C\left(a+|c_F|+\int_XU_+\,\mu_t\right)
\le C\left(a+D_t(F)+\int_Xu_+\,\mu_t\right).
\]
Finally,
\[
u=U-\widetilde Q_F\le \sup_XU+\operatorname{osc}_XQ_F,
\]
and the conclusion follows from \eqref{eq:linear-iteration-bounds}.
\end{proof}
\begin{proposition}\label{prop:surface-inf-green}
There exists a uniform constant $C$ such that
\[
\inf_{y\in X}G_t(x,y)\ge -C
\]
for all $x\in X$ and all $t\ge0$.
\end{proposition}

\begin{proof}
It is enough to consider all sufficiently large $t$, since the estimate on any
bounded time interval follows from standard elliptic theory. Fix $x\in X$ and set
$G_x:=G_t(x,\cdot)$ and $s:=-G_x$. Since $G_x(y)\to+\infty$ as $y\to x$, we have
$s\to-\infty$ near $x$.

Choose smooth convex functions $\beta_\varepsilon:\mathbb R\to\mathbb R_{\ge0}$
such that $0\le\beta_\varepsilon'\le1$, $\beta_\varepsilon''\ge0$,
$\beta_\varepsilon(r)=0$ for $r\le\varepsilon$, and
$0\le r_+-\beta_\varepsilon(r)\le C\varepsilon$. Put
$u_\varepsilon:=\beta_\varepsilon(s)$. Then $u_\varepsilon$ is smooth on $X$,
because it vanishes near the pole. Moreover,
$0\le u_\varepsilon\le s_+\le |G_x|$, and hence
\[
\int_Xu_\varepsilon\,\mu_t\le C
\]
by \cref{prop:surface-L1-green}.

Away from the pole, $\Delta_{g_t}s=-V_t^{-1}$. Using
$\Delta_{g_t}f=2\Delta_{\omega_t}f+
2\langle\nabla f,\theta_t^\#\rangle_{\omega_t}$ and the chain rule we can write
\[
\Delta_{\omega_t}u_\varepsilon
=
\beta_\varepsilon'(s)\Delta_{\omega_t}s
+\beta_\varepsilon''(s)|\partial s|_{\omega_t}^2
\ge
-C-|\nabla u_\varepsilon|_{\omega_t}|\theta_t|_{\omega_t}.
\]
For $\delta>0$, define
\[
F_{\varepsilon,\delta}:=
\left(|\nabla u_\varepsilon|_{\omega_t}^2+\delta^2\right)^{1/2}\Theta_t,
\]
which is a smooth approximation of $|\nabla u_\varepsilon|_{\omega_t}|\theta_t|_{\omega_t}$. Since $\Theta_t\ge|\theta_t|_{\omega_t}$, we have
$\Delta_{\omega_t}u_\varepsilon\ge -C-F_{\varepsilon,\delta}$. By \cref{lem:drifted-laplacian-estimate},
\begin{equation}\label{eq:sup u_epsilon}
    \sup_Xu_\varepsilon
\le
C\left(1+\int_Xu_\varepsilon\,\mu_t+D_t(F_{\varepsilon,\delta})\right).
\end{equation}
Thus it remains to estimate $D_t(F_{\varepsilon,\delta})$.

First,
\[
\Delta_{g_t}u_\varepsilon
=
-\frac{\beta_\varepsilon'(s)}{V_t}
+\beta_\varepsilon''(s)|\nabla s|_{g_t}^2.
\]
Therefore,
\[
\int_X|\nabla u_\varepsilon|_{g_t}^2\,\mu_t
=
-\int_Xu_\varepsilon\Delta_{g_t}u_\varepsilon\,\mu_t
\le
C\int_Xu_\varepsilon\,\mu_t\le C.
\]
Using \cref{eq:Theta-properties}, we can further write
\[
m_t(F_{\varepsilon,\delta})\le Ce^{-t}(1+\delta).
\]
Next set
\[
M_x:=\sup_Xs_+=-\inf_{y\in X}G_t(x,y).
\]
For each $z\in X$, \cref{lem:weighted-energy-identity} gives
\[
\int_X|G_z|\,|\nabla u_\varepsilon|_{g_t}^2\,\mu_t
=
-\int_X|G_z|u_\varepsilon\Delta_{g_t}u_\varepsilon\,\mu_t
+\frac12\int_Xu_\varepsilon^2\Delta_{g_t}|G_z|\,\mu_t.
\]
Using \cref{prop:surface-L1-green} and the bound
\[
-\Delta_{g_t}u_\varepsilon
=
\frac{\beta_\varepsilon'(s)}{V_t}
-\beta_\varepsilon''(s)|\nabla s|_{g_t}^2
\le C,
\]
the first term is bounded above by
$C\int_X|G_z|u_\varepsilon\,\mu_t\le CM_x$. The second term is bounded by
$CM_x^2$ using \cref{lem:laplacian-abs-green}. Hence
\[
\sup_z\int_X|G_z|\,|\nabla u_\varepsilon|_{\omega_t}^2\,\mu_t
\le C(M_x+1)^2.
\]
By the Cauchy-Schwarz inequality and \eqref{eq:Theta-properties},
\[
P_t(F_{\varepsilon,\delta})
\le
\left(\sup_z\int_X|G_z|\Theta_t^2\,\mu_t\right)^{1/2}
\left(
\sup_z\int_X|G_z|\left(|\nabla u_\varepsilon|_{\omega_t}^2+\delta^2\right)\mu_t
\right)^{1/2}
\le
Ce^{-t/2}(M_x+1+\delta).
\]
Consequently,
\[
D_t(F_{\varepsilon,\delta})\le Ce^{-t/2}(M_x+1+\delta).
\]
Substituting this into \eqref{eq:sup u_epsilon} and using
$\int_Xu_\varepsilon\,\mu_t\le C$, we obtain
\[
\sup_Xu_\varepsilon\le C+Ce^{-t/2}(M_x+1+\delta).
\]
Letting $\delta\to0$ and then $\varepsilon\to0$ gives
\[
M_x\le C+Ce^{-t/2}(M_x+1).
\]
For all sufficiently large $t$, the term $Ce^{-t/2}M_x$ is absorbed into the
left-hand side, and so $M_x\le C$. This is equivalent to
\[
\inf_{y\in X}G_t(x,y)\ge -C,
\]
the proof is therefore concluded.
\end{proof}

\subsection{$L^{1+\varepsilon}$-estimate for the Green function}
\begin{proposition}
\label{prop:surface-Lp-green}
There exist uniform constants $\varepsilon>0$ and $C>0$ such that
\[
\sup_{x\in X}\int_X\mathcal G_t(x,y)^{1+\varepsilon}\,\omega_t^2(y)\le C,
\]
where
\[
\mathcal G_t(x,y):=G_t(x,y)-\inf_{w\in X}G_t(x,w)+V_t^{-1}>0.
\]
\end{proposition}

\begin{proof}
It suffices to prove the estimate for all sufficiently large $t$. Fix $x\in X$ and write
$G_x=G_t(x,\cdot)$, $\mathcal G_x=\mathcal G_t(x,\cdot)$ and $\mu_t=\omega_t^2$.
By \cref{prop:surface-L1-green} and \cref{prop:surface-inf-green},
\[
\int_X\mathcal G_x\,\mu_t\le C.
\]
Choose $0<\varepsilon<1/2$. Let $\chi_k:\mathbb R_{\ge0}\to\mathbb R_{\ge0}$ be
smooth nondecreasing functions such that $\chi_k(s)=s$ for $s\le k$,
$\chi_k(s)=k+1$ for $s\ge k+2$, $\chi_k(s)\le s$, and $\chi_k(s)\nearrow s$.
Set $H_k:=\chi_k(\mathcal G_x)$ and $h_k:=H_k^\varepsilon$. Since $H_k$ is
constant near the pole $x$, it is smooth on $X$. Moreover, $0<H_k\le\mathcal G_x$,
$H_k\nearrow\mathcal G_x$, and
\[
\int_Xh_k\,\mu_t\le C.
\]
We also fix the regularized torsion norm $\Theta_t:=\left(|\theta_t|_{\omega_t}^2+e^{-4t}\right)^{1/2}$.

\smallskip
\textbf{Step 1: the real Poisson equation and the estimate of drift terms.}
Let $u_k$ solve the real Poisson equation
\[
\Delta_{g_t}u_k=2(-h_k+\bar h_k),\qquad
\int_Xu_k\,\mu_t=0,\qquad
\bar h_k=V_t^{-1}\int_Xh_k\,\mu_t.
\]
Green's formula gives $u_k(z)=2\int_XG_t(z,y)h_k(y)\,\mu_t(y)$. Since
$G_t(z,\cdot)\ge -C$ and $\int_Xh_k\,\mu_t\le C$, we have $u_k\ge -C$, and hence
\[
\int_X(u_k)_+\,\mu_t\le C.
\]
Let $U_k:=\sup_Xu_k$. From $|G_t(z,y)|\le G_t(z,y)+C$ and the formula above,
\[
P_t(h_k)=\sup_z\int_X|G_z|h_k\,\mu_t\le C(U_k+1).
\]
Thus \cref{lem:real-poisson-green-norm} gives
\[
\sup_z\int_X|G_z|\,|\nabla u_k|_{\omega_t}^2\,\mu_t
+
\int_X|\nabla u_k|_{\omega_t}^2\,\mu_t
\le C(U_k+1)^2.
\]

For $\eta>0$, define smooth functions
\[
F_{k,\eta}:=\delta_0
\left(|\nabla u_k|_{\omega_t}^2+\eta^2\right)^{1/2}\Theta_t,
\]
where $0<\delta_0$ will be fixed below. We record the uniform estimate
\begin{equation}\label{eq:Dt F keta}
    D_t(F_{k,\eta})\le Ce^{-t/2}(U_k+1+\eta).
\end{equation}
Indeed, using $\sqrt{a^2+\eta^2}\le a+\eta$, \eqref{eq:Theta-properties}, the Cauchy-Schwarz inequality, and the estimates above,
\[
\begin{aligned}
m_t(F_{k,\eta})
&\le
C\left(\int_X|\nabla u_k|_{\omega_t}^2\,\mu_t\right)^{1/2}
\left(\int_X\Theta_t^2\,\mu_t\right)^{1/2}
+C\eta\int_X\Theta_t\,\mu_t  \\
&\le
Ce^{-t}(U_k+1)+C\eta e^{-t}
\le Ce^{-t/2}(U_k+1+\eta),
\end{aligned}
\]
and similarly 
\[
\begin{aligned}
P_t(F_{k,\eta})
&\le
C\left(
\sup_z\int_X|G_z|\,|\nabla u_k|_{\omega_t}^2\,\mu_t
\right)^{1/2}
\left(
\sup_z\int_X|G_z|\Theta_t^2\,\mu_t
\right)^{1/2}  \\
&\quad
+C\eta
\left(
\sup_z\int_X|G_z|\,\Theta_t^2\,\mu_t
\right)^{1/2}
\left(
\sup_z\int_X|G_z|\,\mu_t
\right)^{1/2}  \\
&\le
Ce^{-t/2}(U_k+1+\eta).
\end{aligned}
\]
\smallskip
\textbf{Step 2: the uniform upper bound of $u_k$.}
Solve
\begin{equation}\label{eq:auxiliary MA}
    (\hat\omega_t+\sqrt{-1}\partial\bar\partial\psi_{t,k})^2
=
c_k\frac{H_k^{2\varepsilon}+1}{B_k}\,\omega_t^2,\qquad
\sup_X\psi_{t,k}=0,
\end{equation}
where $B_k=\int_X(H_k^{2\varepsilon}+1)\,\mu_t$ and the dependence of $c_k$ on $t$ was omitted. The existence of such solutions is guaranteed by the main theorem of \cite{Yau78, TW10}. Since $2\varepsilon<1$ and
$\int_X\mathcal G_x\,\mu_t\le C$, the quantities $B_k$ are uniformly bounded
above and below. The constants $c_k$ are uniformly bounded above and below by the
bounded mass and positive volume properties. Indeed, thanks to the fact that $X$ is K\"ahler, integrating both sides of \eqref{eq:auxiliary MA} we obtain
\[
0<\underline{\text{Vol}}(\chi)\le c_k\le\overline{\text{Vol}}(A\omega_0)<+\infty.
\]
As in \cite[Lemma 5.3]{Sun26}, one can check that the quantity 
$$
\int_X\frac{H_{t,k}^{2\varepsilon}+1}{B_k}e^{F_t}\log\left(1+\frac{H_{t,k}^{2\varepsilon}+1}{B_k}e^{F_t}\right)^p\omega_X^2,
$$
is uniformly bounded from above. Hence the uniform $L^\infty$ estimate (cf. \cite[Theorem 15.4]{PSWZ25}) gives
\[
\|\psi_{t,k}-\varphi_t\|_{L^\infty(X)}\le C.
\]
By the arithmetic-geometric mean inequality,
\[
\operatorname{tr}_{\omega_t}(\hat\omega_t+\sqrt{-1}\partial\bar\partial\psi_{t,k})
\ge
2\left(\frac{c_k}{B_k}\right)^{1/2}(H_k^{2\varepsilon}+1)^{1/2}
\ge a_0h_k
\]
for a uniform $a_0>0$. Hence
\[
\Delta_{\omega_t}(\psi_{t,k}-\varphi_t)\ge a_0h_k-C.
\]
Fix $0<\delta_0<a_0/2$ and set
\[
v_k:=
(\psi_{t,k}-\varphi_t)
-\frac1{V_t}\int_X(\psi_{t,k}-\varphi_t)\,\mu_t
+\delta_0u_k.
\]
Since $\Delta_{\omega_t}u_k
=
-h_k+\bar h_k-\langle\nabla u_k,\theta_t^\#\rangle_{\omega_t}$, we obtain
\[
\Delta_{\omega_t}v_k
\ge
(a_0-\delta_0)h_k-C
-\delta_0\langle\nabla u_k,\theta_t^\#\rangle_{\omega_t}
\ge
-C-F_{k,\eta},
\]
where the last inequality uses $\Theta_t\ge|\theta_t|_{\omega_t}$ and the definition
of $F_{k,\eta}$. Applying \cref{lem:drifted-laplacian-estimate} to $v_k$ gives
\[
\sup_Xv_k\le
C\left(1+\int_X(v_k)_+\,\mu_t+D_t(F_{k,\eta})\right).
\]
Since $\|\psi_{t,k}-\varphi_t\|_{L^\infty}\le C$ and
$\int_X(u_k)_+\,\mu_t\le C$, we have $\int_X(v_k)_+\,\mu_t\le C$. Therefore, by \eqref{eq:Dt F keta} we have
\[
\sup_Xv_k\le C+Ce^{-t/2}(U_k+1+\eta).
\]
Letting $\eta\to0$ yields
\[
\sup_Xv_k\le C+Ce^{-t/2}(U_k+1).
\]
Let $z_k$ be a maximum point of $u_k$. Since $\psi_{t,k}-\varphi_t$ and its average
are uniformly bounded, $v_k(z_k)\ge\delta_0U_k-C$. Hence
\[
\delta_0U_k-C\le C+Ce^{-t/2}(U_k+1).
\]
For all sufficiently large $t$, the term $Ce^{-t/2}U_k$ is absorbed into the
left-hand side, and we obtain
\[
U_k\le C.
\]
\smallskip
\textbf{Step 3: conclusion.}
By Green's formula,
\[
u_k(x)=2\int_XG_xh_k\,\mu_t.
\]
Write $G_x=\mathcal G_x-A_x$, where $A_x=-\inf_yG_x(y)+V_t^{-1}\le C$. Then
\[
u_k(x)
=
2\int_X\mathcal G_xH_k^\varepsilon\,\mu_t
-
2A_x\int_XH_k^\varepsilon\,\mu_t.
\]
The second term is bounded below uniformly, because
$A_x\le C$ and $\int_XH_k^\varepsilon\,\mu_t=\int_Xh_k\,\mu_t\le C$. Since
$u_k(x)\le U_k\le C$, it follows that
\[
\int_X\mathcal G_xH_k^\varepsilon\,\mu_t\le C.
\]
Letting $k\to\infty$, we have $H_k\nearrow\mathcal G_x$, and the monotone
convergence theorem gives
\[
\int_X\mathcal G_x^{1+\varepsilon}\,\mu_t\le C.
\]
The estimate is uniform in $x$, and the proof is complete.
\end{proof}

\section{End of proof of the diameter and volume estimates}\label{section 5}

We now finish the proof of the diameter estimate and the volume non-collapsing
estimate. At this stage the argument is purely Riemannian, following almost word for word from
\cite{GPSS24a, Sun26}. 

First we record the gradient estimate for the Green function.

\begin{lemma}\label{lem:gradient-estimate-green}
Let $\varepsilon>0$ be the constant in \cref{prop:surface-Lp-green}. Then for any
\[
1<s<\frac{2+2\varepsilon}{2+\varepsilon},
\]
there exists a uniform constant $C_s>0$ such that
\[
\sup_{x\in X}\int_X|\nabla_yG_t(x,y)|_{g_t}^s\,\mu_t(y)\le C_s.
\]
\end{lemma}

\begin{proof}
Fix $x\in X$ and write
\[
\mathcal G_x(y):=\mathcal G_t(x,y)
=
G_t(x,y)-\inf_{w\in X}G_t(x,w)+V_t^{-1}>0.
\]
Then $\mathcal G_x\ge V_t^{-1}\ge C^{-1}$ and
\(\nabla_y\mathcal G_x=\nabla_yG_t(x,y)\). We first prove that for every $\beta>0$,
\[
\int_X
\frac{|\nabla_yG_t(x,y)|_{g_t}^2}{\mathcal G_x(y)^{1+\beta}}\,\mu_t(y)
\le C\beta^{-1}.
\]
Indeed, since
\[
\Delta_{g_t}\mathcal G_x\,\mu_t
=
-\delta_x+\frac1{V_t}\mu_t
\]
and $\mathcal G_x^{-\beta}(x)=0$, integration by parts gives
\[
\begin{aligned}
\beta\int_X\mathcal G_x^{-1-\beta}|\nabla\mathcal G_x|_{g_t}^2\,\mu_t
&=
-\int_X\langle\nabla(\mathcal G_x^{-\beta}),\nabla\mathcal G_x\rangle_{g_t}\,\mu_t  \\
&=
\int_X\mathcal G_x^{-\beta}\Delta_{g_t}\mathcal G_x\,\mu_t
=
\frac1{V_t}\int_X\mathcal G_x^{-\beta}\,\mu_t
\le C.
\end{aligned}
\]
The integration by parts can be justified by cutting off a small ball around the pole and
then letting its radius tend to zero, as was done in \cite{GPSS24a} and \cite{Sun26}.

Now choose $\beta>0$ such that
\[
\frac{(1+\beta)s}{2-s}=1+\varepsilon.
\]
This is possible precisely when
\[
s<\frac{2+2\varepsilon}{2+\varepsilon}.
\]
By H\"older's inequality and \cref{prop:surface-Lp-green},
\[
\begin{aligned}
\int_X|\nabla G_t(x,\cdot)|_{g_t}^s\,\mu_t
&=
\int_X
\left(
\frac{|\nabla G_t(x,\cdot)|_{g_t}^2}{\mathcal G_x^{1+\beta}}
\right)^{s/2}
\mathcal G_x^{(1+\beta)s/2}\,\mu_t  \\
&\le
\left(
\int_X
\frac{|\nabla G_t(x,\cdot)|_{g_t}^2}{\mathcal G_x^{1+\beta}}\,\mu_t
\right)^{s/2}
\left(
\int_X
\mathcal G_x^{\frac{(1+\beta)s}{2-s}}\,\mu_t
\right)^{(2-s)/2} \\
&\le C_s.
\end{aligned}
\]
The estimate is uniform in $x$ and $t$.
\end{proof}

\begin{theorem}\label{thm:diameter-volume-estimates}
There exist uniform constants $C>0$, $c>0$, $\alpha>0$ and $r_0>0$ such that
\[
\operatorname{diam}(X,\omega_t)\le C
\]
and
\[
\operatorname{Vol}_{\omega_t}\big(B_{\omega_t}(x,r)\big)\ge cr^\alpha
\]
for all $x\in X$, $0<r<r_0$, and all $t\ge0$.
\end{theorem}

\begin{proof}
It is enough to prove the estimates for all sufficiently large $t$, since on every bounded
time interval the metrics are uniformly smooth equivalent.

We first prove the diameter bound. Fix $x_0\in X$ and set
$d(y):=d_{g_t}(x_0,y)$. The function $d$ is Lipschitz and $|\nabla d|_{g_t}\le1$
a.e. Applying Green's formula by a standard smooth approximation of $d$, we get
\[
0=d(x_0)=\frac1{V_t}\int_Xd\,\mu_t
+
\int_X\langle\nabla_yG_t(x_0,y),\nabla d(y)\rangle_{g_t}\,\mu_t(y).
\]
Therefore, by \cref{lem:gradient-estimate-green},
\[
\frac1{V_t}\int_Xd\,\mu_t
\le
\int_X|\nabla_yG_t(x_0,y)|_{g_t}\,\mu_t(y)
\le C.
\]
For any $y_0\in X$, applying Green's formula to $d$ at $y_0$ gives
\[
\begin{aligned}
d_{g_t}(x_0,y_0)
&=
\frac1{V_t}\int_Xd\,\mu_t
+
\int_X\langle\nabla_yG_t(y_0,y),\nabla d(y)\rangle_{g_t}\,\mu_t(y) \\
&\le
C+\int_X|\nabla_yG_t(y_0,y)|_{g_t}\,\mu_t(y)
\le C.
\end{aligned}
\]
Taking the supremum over $x_0,y_0\in X$ proves the diameter estimate.

We next prove the volume non-collapsing estimate. Fix
\[
1<s<\frac{2+2\varepsilon}{2+\varepsilon}
\]
and set
\[
\alpha:=\frac{s}{s-1}.
\]
Let $x\in X$ and $0<r<r_0$, where $r_0$ is chosen small and fixed. If
$B_{g_t}(x,r)=X$, the desired estimate follows from the uniform lower bound of $V_t$.
Otherwise, choose $z\in X\setminus B_{g_t}(x,r)$. Let $\eta$ be a Lipschitz cutoff function
supported in $B_{g_t}(x,r)$, satisfying $\eta\equiv1$ on $B_{g_t}(x,r/2)$ and
$|\nabla\eta|_{g_t}\le Cr^{-1}$. Set $d_x(y):=d_{g_t}(x,y)$.

Since $d_x\eta(z)=0$, Green's formula gives
\[
\frac1{V_t}\int_X d_x\eta\,\mu_t
=
-\int_X\langle\nabla_yG_t(z,y),\nabla(d_x\eta)(y)\rangle_{g_t}\,\mu_t(y).
\]
As $d_x\le r$ on the support of $\eta$ and $|\nabla d_x|\le1$ a.e.,
\[
|\nabla(d_x\eta)|_{g_t}\le C\mathbf 1_{B_{g_t}(x,r)}.
\]
Hence
\[
\frac1{V_t}\int_X d_x\eta\,\mu_t
\le
C\int_{B_{g_t}(x,r)}|\nabla_yG_t(z,y)|_{g_t}\,\mu_t(y)
\le
C\operatorname{Vol}_{\omega_t}(B_{g_t}(x,r))^{(s-1)/s},
\]
where we used \cref{lem:gradient-estimate-green} in the last inequality.

Choose $\hat z\in \partial B_{g_t}(x,r/2)$ when the boundary is nonempty. If it is empty,
then $B_{g_t}(x,r/2)=X$ and the desired estimate is trivial. Since
$d_x\eta(\hat z)=r/2$, another application of Green's formula gives
\[
\begin{aligned}
\frac r2
&=
d_x\eta(\hat z)  \\
&=
\frac1{V_t}\int_Xd_x\eta\,\mu_t
+
\int_X\langle\nabla_yG_t(\hat z,y),\nabla(d_x\eta)(y)\rangle_{g_t}\,\mu_t(y) \\
&\le
C\operatorname{Vol}_{\omega_t}(B_{g_t}(x,r))^{(s-1)/s}.
\end{aligned}
\]
Therefore
\[
\operatorname{Vol}_{\omega_t}(B_{g_t}(x,r))
\ge
cr^{s/(s-1)}
=
cr^\alpha.
\]
This proves the volume non-collapsing estimate.
\end{proof}

\section{Identification of the Gromov-Hausdorff limit}\label{section 6}

Let $E=\text{Null}(K_X)$ and let
\[
f:X\longrightarrow Y:=X_{\rm can}
\]
be the canonical map induced by the linear system $|mK_X$| for some large $m$. We denote by
\[
E_Y:=f(E)=Y_{sing}
\]
the finite set of orbifold points of $Y$ (cf. \cite{BHP04}). Let $\omega_{KE}$ be the orbifold
K\"ahler-Einstein metric on $Y$ and let $d_{can}$ be the induced metric completion
distance. By a result of Song \cite{Song14}, we already know that $\overline{(X\setminus E,\omega_{KE})}$ is homeomorphic to $X_{can}$.

As in \cite[Lemma 4.3]{LTZ26}, (see also \cite[Section 9]{Sun26}), using the uniform diameter estimate, the volume
non-collapsing estimate \cref{thm:diameter-volume-estimates}, and the bound $\omega_t^2\le C\Omega$ \cref{lem:basic-flow-estimates}, the proof of the
Gromov-Hausdorff convergence reduces to the following distance comparison: for
every $\varepsilon>0$ and all $p,q\in X\setminus\widetilde V_\varepsilon$,
\begin{equation}\label{eq:distance-comparison-needed}
d_{can}(f(p),f(q))
\le
d_t(p,q)+\Psi(t^{-1}\mid\varepsilon),
\end{equation}
where
\[
V_\varepsilon:=\{y\in Y:\ d_{can}(y,E_Y)<\varepsilon\},
\qquad
\widetilde V_\varepsilon:=f^{-1}(V_\varepsilon).
\]
Here and below $\Psi(t^{-1}\mid\varepsilon)\to0$ as $t\to\infty$ for fixed $\varepsilon$.

\begin{lemma}\label{lem:canonical-distance-upper}
For every fixed sufficiently small $\varepsilon>0$ and all
$p,q\in X\setminus\widetilde V_\varepsilon$, we have
\[
d_{can}(f(p),f(q))
\le
d_t(p,q)+\Psi(t^{-1}\mid\varepsilon)+\Psi(\varepsilon).
\]
Consequently, \eqref{eq:distance-comparison-needed} holds after letting
$\varepsilon\to0$ in the reduction above.
\end{lemma}

\begin{proof}
Fix $\varepsilon>0$ small enough so that the connected components
$V_{\varepsilon,\nu}$ of $V_\varepsilon$ are mutually disjoint. Since $E_Y$ is a discrete
finite set, there are only finitely many such components, and
\[
\max_\nu\operatorname{diam}_{d_{can}}(V_{\varepsilon,\nu})
\le \Psi(\varepsilon).
\]
For notational simplicity, we introduce the auxiliary pseudo-distance $d_{can,\varepsilon}$ on
$Y\setminus V_\varepsilon$ by forcing each component $V_{\varepsilon,\nu}$ to
have zero diameter. Equivalently, for $a,b\in Y\setminus V_\varepsilon$,
$d_{can,\varepsilon}(a,b)$ is the infimum of the $d_{can}$-length of the part of
a piecewise smooth curve from $a$ to $b$ lying in $Y\setminus V_\varepsilon$.
Then we clearly have
\begin{equation}\label{eq:quotient-distance-comparison}
d_{can,\varepsilon}(a,b)\le d_{can}(a,b)\le d_{can,\varepsilon}(a,b)+\Psi(\varepsilon).
\end{equation}
Now fix $p,q\in X\setminus\widetilde V_\varepsilon$. Since
$X\setminus\widetilde V_\varepsilon\Subset X\setminus D$, the local smooth
convergence \cref{lem:basic-flow-estimates} gives, for $t$ sufficiently large,
\[
f^*\omega_{KE}\le (1+\Psi(t^{-1}\mid\varepsilon))\,\omega_t
\quad\text{on }X\setminus\widetilde V_\varepsilon.
\]
Let $\gamma$ be a piecewise smooth curve joining $p$ to $q$. In the
$d_{can,\varepsilon}$-length of $f\circ\gamma$, the portions lying inside
$V_\varepsilon$ are ignored. On the remaining portions, $\gamma$ lies in
$X\setminus\widetilde V_\varepsilon$, and hence the previous metric comparison gives
\[
L_{can,\varepsilon}(f\circ\gamma)
\le
(1+\Psi(t^{-1}\mid\varepsilon))L_t(\gamma).
\]
Taking the infimum over all such $\gamma$, we obtain
\[
d_{can,\varepsilon}(f(p),f(q))
\le
(1+\Psi(t^{-1}\mid\varepsilon))d_t(p,q).
\]
Since $\operatorname{diam}(X,\omega_t)\le C$, this implies
\[
d_{can,\varepsilon}(f(p),f(q))
\le
d_t(p,q)+\Psi(t^{-1}\mid\varepsilon).
\]
Combining this with \eqref{eq:quotient-distance-comparison} yields
\[
d_{can}(f(p),f(q))
\le
d_t(p,q)+\Psi(t^{-1}\mid\varepsilon)+\Psi(\varepsilon),
\]
as required.
\end{proof}

\begin{theorem}\label{thm:GH-convergence-canonical-model}
The metric spaces $(X,d_t)$ converge in the Gromov-Hausdorff topology to
$(Y,d_{can})$ as $t\to+\infty$.
\end{theorem}

\begin{proof}
By the above reduction, it remains to verify
\eqref{eq:distance-comparison-needed}. This is exactly
\cref{lem:canonical-distance-upper}. Indeed, for fixed $\varepsilon>0$, the lemma
gives
\[
d_{can}(f(p),f(q))
\le
d_t(p,q)+\Psi(t^{-1}\mid\varepsilon)+\Psi(\varepsilon)
\]
for all $p,q\in X\setminus\widetilde V_\varepsilon$. Letting first $t\to\infty$ and
then $\varepsilon\to0$ gives the required distance comparison.
\end{proof}

\end{document}